\documentclass[10pt,reqno]{amsart}

\usepackage{amsmath,amssymb,amsfonts,latexsym,amsthm}

\usepackage{mathrsfs}

\numberwithin{equation}{section}

\usepackage{graphicx}

\usepackage{ifthen} % pacchetto necessario per un minimo di
\calclayout
\provideboolean{shownotes} % dichiariamo una variabile booleana
\setboolean{shownotes}{false} % la definiamo come "vera" o "falsa"
\newcommand{\margnote}[1]{
\ifthenelse{\boolean{shownotes}}%
{\marginpar{\raggedright\tiny\texttt{#1}}}%
{}%
}

\newcommand{\hole}[1]{
\ifthenelse{\boolean{shownotes}}%
{\begin{center} \fbox{ \rule {.25cm}{0cm}
\rule[-.1cm]{0cm}{.4cm} \parbox{.85\textwidth}{\begin{center}
\texttt{#1}\end{center}} \rule {.25cm}{0cm}}\end{center}}
{}
}

\theoremstyle{plain}
\newtheorem{defi}{Definition}[section]
\newtheorem{theo}{Theorem}[section]
\newtheorem{lema}{Lemma}

\newtheorem{rem}{Remark}
\theoremstyle{remark}

\theoremstyle{remark}

\begin{document}

\title[]{Hyperbolic systems of quasilinear equations in compressible fluid dynamics with an objective Cattaneo-type extension for the heat flux}

\author[F. Angeles]{Felipe Angeles}
 
\address{{\rm (F. Angeles)} Instituto de Matem\'aticas\\Universidad Nacional Aut\'{o}noma 
de M\'{e}xico\\Circuito Exterior s/n, Ciudad de M\'{e}xico C.P. 04510 (Mexico)}

\email{teojkd@ciencias.unam.mx}

\begin{abstract}
We consider the coupling between the equations of motion of an inviscid compressible fluid in space with an objective Cattaneo-type extension for the heat flux. These equations are written in quasilinear form and we determine which of the given formulations for the heat flux allows for the hyperbolicity of the system. This feature is necessary for a physically acceptable sense of well-posedness for the Cauchy problem of such system of equations.
\end{abstract}
\maketitle

\setcounter{tocdepth}{1}

\section{Introduction}
One of the best known substitutes for Fourier's law of heat conduction in continuum thermodynamics is the Maxwell-Cattaneo law \cite{cat}
\begin{align}
\tau q_{t}+q=-\kappa\nabla\theta,\label{eq:MC}
\end{align}
where $q$ is the heat flux, $\theta$ the temperature field and $\kappa$ stands as the thermal conductivity. Although the Maxwell-Cattaneo law accounts for finite speed of heat conduction, this model is not Galilean invariant \cite{chjo}. By replacing the partial time derivative in \eqref{eq:MC} with a material derivative (see \eqref{eq:4}) Christov and Jordan showed that a Galilean invariant formulation of the Maxwell-Cattaneo law is obtained, one that predicts the finite speed of propagation of heat \cite{chjo}. However, they also showed that the heat flux in this model cannot be eliminated in order to obtain a single equation for the temperature field. In \cite{christov}, Christov argues the importance of replacing the partial time derivative in \eqref{eq:MC} with an \emph{objective} derivative and proposes 
\begin{equation}
\tau\left[q_{t}+v\cdot\nabla q-q\cdot\nabla v+(\nabla\cdot v)q\right]+q=-\kappa\nabla\theta.\label{eq:5}
\end{equation}
Moreover, Christov shows that, when this evolution equation is combined with the material invariant form of the balance of internal energy, it allows for the heat flux to be eliminated in order to obtain a single hyperbolic equation for the temperature. By considering the coupling between the local form of the conservation of mass ($\rho$), the balance of linear momentum ($\rho v$) and the balance of total energy ($E=\frac{1}{2}\rho|u|^{2}+e$) for a compressible inviscid fluid in space, that is, 
\begin{eqnarray}
\rho_{t}+\nabla\cdot(\rho v)&=&0,\label{eq:1}\\
(\rho v)_{t}+\nabla\cdot(\rho v\otimes v)&=&-\nabla\cdot p,\label{eq:2}\\
(\rho E)_{t}+\nabla\cdot(\rho E v)&=&-\nabla\cdot(pv)-\nabla\cdot q,\label{eq:3}
\end{eqnarray}
together with \eqref{eq:5}, Straughan showed that, for a given wavenumber $\xi_{0}$, there are values of $q$ for which an acoustic wave propagates together with a thermal wave and completely determined the wavespeed for the purely thermal and purely mechanical cases \cite{stra}. However, it was recently shown that this system of equations, also known as the Cattaneo-Christov system, is not \emph{hyperbolic} \cite{felipe1}. The notion of hyperbolicity for a quasilinear systems of $N$ equations is related with the possibility of finding $N$ \emph{different} (linearly independent) waves propagating in any spatial direction (cf. \cite[Chapter III]{daf}). In \cite[Section 5]{felipe1}, it is shown that, for the Cattaneo-Christov system, we can always find particular directions $\xi$ and values of $q$ for which the $N$ different waves doesn't exist. In particular, since the \emph{characteristic speeds} of this system are real, it can be classified as \emph{weakly hyperbolic}. Although the Cauchy problem for weakly hyperbolic systems of equations can be well-posed in Grevey spaces, it is not a sufficient condition for the $\mathcal{C}^{\infty}$-well-posedness (see \cite{colombinimetivier} and the references therein). Moreover, it is well known that the \emph{hyperbolicity} (cf. \cite{benzoni}, \cite{daf}, \cite{serre}) of a first order quasilinear system of equations is a necessary condition for the existence of $L^{2}$-\emph{energy estimates} (cf. \cite[Theorem 3.1.2 and Lemma 3.1.3]{serre}, \cite{kajitani} and \cite{kano}).\\
For these reasons, the present work considers the coupling between \eqref{eq:1}-\eqref{eq:3} with 
\begin{equation}
\tau\left[\partial_{t}q+v\cdot\nabla q-\frac{1}{2}\left(\nabla v-\nabla v^{\top}\right)q+\frac{\lambda}{2}\left(\nabla v+\nabla v^{\top}\right) q+\nu(\nabla\cdot v)q\right]+q=-\kappa\nabla\theta,
\label{eq:objder}
\end{equation}
In \cite{morro2}, Morro shows that \eqref{eq:objder} is objective for any pairs of invariant scalars $\lambda,\nu$. Using an objective derivative for the heat flux and questioning if such formulation is compatible with thermodynamics is an issue that has been addressed in detail by Morro and Giorgi (see \cite{amorro}, \cite{morro2}, \cite{amorro2}, \cite{mogi} for example). In contrast, this work determines which of the frame indifferent formulations for the heat flux in \eqref{eq:objder}, yields a hyperbolic system of quasilinear equations when it is coupled with \eqref{eq:1}-\eqref{eq:3}. We show that there is only one heat flux model in \eqref{eq:objder} with such feature, namely, when $(\lambda,\nu)=(1,-1)$.
\section{Thermodynamical assumptions}
Throughout this paper we make the following thermodynamical and constitutive assumptions:
\begin{itemize}
\item [(\textbf{C1})] The independent thermodynamical variables are the mass density field $\rho(x,t)>0$ and the absolute temperature field $\theta(x,t)>0$. They vary within the domain $\mathcal{D}:=\left\lbrace(\rho,\theta)\in\mathbb{R}^{2}:\rho>0,\theta>0\right\rbrace$.
	The thermodynamic pressure $p$, the internal energy $e$ and the thermal conductivity $\kappa$ are given smooth functions of $(\rho,\theta)\in\mathcal{D}$.
	\item [(\textbf{C2})] The fluid satisfies the following conditions $p, p_{\rho}, p_{\theta}, e_{\theta},\kappa>0$ for all $(\rho,\theta)\in\mathcal{D}$.
\end{itemize}
In particular, assumption (\textbf{C2}) refers to compressible fluids satisfying the standard assumptions of Weyl \cite{weyl}. Also, we assume that, 
\begin{itemize}
	\item [(\textbf{C3})] $\lambda$ and $\nu$ are real valued functions of $(\rho,\theta)\in\mathcal{D}$. 
\end{itemize}
\section{Hyperbolic quasilinear systems of equations}
Let $\lambda$ and $\nu$ be given functions of $(\rho,\theta)\in\mathcal{D}$. Consider the state variable $U=(\rho,v,\theta,q)^{\top}\in\mathcal{O}\subset\mathbb{R}^{N}$, where $\mathcal{O}:=\left\lbrace(\rho,v,\theta,q)\in\mathbb{R}^{N}:\rho>0,\theta>0\right\rbrace$, $N=2d+2$ and $d$ is the spatial dimension. Then, the quasilinear form of \eqref{eq:1}-\eqref{eq:objder} is,
\begin{equation}
A^{0}(U)U_{t}+A^{i}_{\lambda\nu}(U)\partial_{i}U+Q(U)=0,
\label{eq:21}
\end{equation}
where repeated index notation has been used in the space derivatives $\partial_{i}$ and $i=1,..,d$. Here, once $U\in\mathcal{O}$ is given, each coefficient $A^{0}(U)$, $A^{i}(U)$ is a matrix of order $N\times N$ and $Q(U)$ is a vector in $\mathbb{R}^{N}$. In the subsequence, we will refer to \eqref{eq:21} as the \emph{induced} $(\lambda,\nu)$-quasilinear system of equations by the  $(\lambda,\nu)$-objective heat flux model \eqref{eq:objder}.
For $\xi=(\xi_{1},..,\xi_{d})\in\mathbb{S}^{d-1}$ and $U\in\mathcal{O}$ we define the symbol 
\begin{equation}
A_{\lambda\nu}(\xi;U):=\sum_{i=1}^{3}A^{i}_{\lambda,\nu}(U)\xi_{i}.\label{eq:22}
\end{equation}
Let us recall the definition of hyperbolic quasilinear system of equations (cf. \cite{benzoni}, \cite{daf}, \cite{serre}).
\begin{defi}[Hyperbolicity]
\label{hyperbolic}
A quasilinear system of the form \eqref{eq:21} is called hyperbolic if for any fixed state $U\in\mathcal{O}$ and $\xi\in\mathbb{S}^{d-1}$ the matrix $A^{0}(U)$ is non singular and the eigenvalue problem 
\begin{align}
\left(A_{\lambda\nu}(\xi;U)-\eta A^{0}(U)\right)V=0\label{eq:hypdet}
\end{align} 
has real eigenvalues ($\eta\in\mathbb{R}$) and $N$ linearly independent eigenvectors. In particular, the eigenvalues of \eqref{eq:hypdet}, also known as the characteristic speeds of \eqref{eq:21}, satisfy the equation
\begin{align}
\operatorname{det}\left(A_{\lambda\nu}(\xi;U)-\eta A^{0}\right)=0,\label{eq:cs}
\end{align}
for each $U\in\mathcal{O}$ and $\xi\in\mathbb{S}^{d-1}$. Therefore, $\eta=\eta(\xi;U)$. When $d=1$ we simply write $\eta=\eta(U)$.
\end{defi}

\section{One dimensional system}
In the one dimensional case, the matrices defining system \eqref{eq:21} are
\begin{align*}
A^{0}(U):=\left(\begin{array}{cccc}
	1&0&0&0\\
	0&\rho&0&0\\
	0&0&\rho e_{\theta}&0\\
	0&0&0&\tau
\end{array}\right)&,\quad A^{1}(U):=\left(\begin{array}{cccc}
	v&\rho&0&0\\
	p_{\rho}&\rho v& p_{\theta}&0\\
	0&\theta p_{\theta}&\rho ve_{\theta}&1\\
	0&(\lambda+\nu)q&\kappa&\tau v
\end{array}\right),\\
Q(U)&:=\left(0,0,0,q\right)^{T}.
\end{align*}
First, we seek the roots of \eqref{eq:cs}. We use the formula, 
\begin{align}
\label{eq:detblock}
\operatorname{det}\left(\begin{array}{cc}
L&M\\
N&P	
\end{array}\right)=\left(\operatorname{det}L\right)\operatorname{det}\left(P-NL^{-1}M\right),
\end{align}
whenever $L$ is invertible (see \cite{fuz} for example). In this case,
\begin{align*}
L=\left(\begin{array}{cc}
	v-\eta&\rho\\
	P_{\rho}&\rho(v-\eta)
\end{array}\right),\quad M=\left(\begin{array}{cc}
	0&0\\
	p_{\theta}&0
\end{array}\right),\quad N=\left(\begin{array}{cc}
	0&\theta p_{\theta}\\
	0&(\lambda+\nu)q
\end{array}\right),\quad P=\left(\begin{array}{cc}
	\rho e_{\theta}(v-\eta)&1\\
	\kappa&\tau(v-\eta)
\end{array}\right),
\end{align*}
and $\operatorname{det}L=\rho(v-\eta)^{2}-\rho p_{\rho}$. Then, according with \eqref{eq:detblock}, \eqref{eq:cs} is equivalent to 
\begin{align}
\rho e_{\theta}\tau(v-\eta)^{4}-\left(\rho e_{\theta}p_{\rho}\tau+\frac{\theta p_{\theta}^{2}\tau}{\rho}+\kappa\right)(v-\eta)^{2}+(\lambda+\nu)\frac{p_{\theta}q}{\rho}(v-\eta)+\kappa p_{\rho}=0.\label{eq:cp1d}
\end{align}
If we set $z:=v-\eta$, and multiply by $(\rho e_{\theta}\tau)^{-1}$ in \eqref{eq:cp1d}, we obtain 
\begin{align}
z^{4}+a_{2}z^{2}+a_{1}z+a_{0}=0,\label{eq:dquartic}
\end{align}
where 
\begin{align}
\label{eq:coefdep}
a_{2}=-\frac{1}{\rho e_{\theta}\tau}\left(\rho e_{\theta}p_{\rho}\tau+\frac{\theta p_{\theta}^{2}\tau}{\rho}+\kappa\right),\quad a_{1}=(\lambda+\nu)\frac{p_{\theta}}{\rho^{2} e_{\theta}\tau}q,\quad a_{0}=\frac{\kappa p_{\rho}}{\rho e_{\theta}\tau}.
\end{align}
First, we show that, if $\lambda+\nu\neq 0$, the one dimensional case of \eqref{eq:21}, is not a hyperbolic system of equations. We use the following result stated without a proof (see \cite[Theorem 7]{eug} and also \cite{rees}, for example).
\begin{lema}
\label{fourthroots}
A quartic equation of the form \eqref{eq:dquartic} with $a_{0}$, $a_{1}$, $a_{2}$ real, $a_{1}\neq 0$, and with discriminant $\Delta$, has 2 distinct real roots and 2 imaginary roots if $\Delta<0$. \qed
\end{lema}
The discriminant, $\Delta=\Delta(a_{0},a_{1},a_{2})$, of \eqref{eq:dquartic} is given as 
\begin{align}
\Delta(a_{0},a_{1},a_{2})=256a_{0}^{3}-128a_{2}^{2}a_{0}^{2}+16a_{0}a_{2}^{4}+144a_{0}a_{1}^{2}a_{2}-4a_{1}^{2}a_{2}^{3}-27a_{1}^{4},\label{eq:discriminant}
\end{align}
see \cite[page 41]{eug} and \cite[page 405]{geka} for example.
\begin{theo}
\label{complexroots}
Let $(\rho^{\ast},\theta^{\ast})\in\mathcal{D}$ be such that $\lambda(\rho^{\ast},\theta^{\ast})+\nu(\rho^{\ast},\theta^{\ast})\neq 0$. Then, there are values of $q\in\mathbb{R}$ for which the characteristic polynomial \eqref{eq:cp1d} has complex roots and thus, the $(\lambda,\nu)$-quasilinear system \eqref{eq:21} is not hyperbolic.
\end{theo}
\begin{proof}
Set $\gamma^{\ast}:=\lambda(\rho^{\ast},\theta^{\ast})+\nu(\rho^{\ast},\theta^{\ast})$. Observe that \eqref{eq:discriminant} can be rewritten as a fourth order polynomial in the variable $a_{1}$, namely
\begin{align*}
\mathcal{P}(a_{1}):=Aa_{1}^{4}+Ba_{1}^{2}+C,
\end{align*}
where $A:=-27$, $B:=144a_{0}a_{2}-4a_{2}^{3}>0$, $C:=16a_{0}\left(a_{2}^{2}-4a_{0}\right)^{2}\geq 0$ and, according with \eqref{eq:coefdep}, $B=B(\rho^{\ast},\theta^{\ast})$, $C=C(\rho^{\ast},\theta^{\ast})$ and $a_{1}=g(\rho^{\ast},\theta^{\ast})q$ for $g(\rho^{\ast},\theta^{\ast})=\gamma^{\ast}\tfrac{p_{\theta}(\rho^{\ast},\theta^{\ast})}{(\rho^{\ast})^{2}e_{\theta}(\rho^{\ast},\theta^{\ast})\tau}\neq 0$. By taking,
\begin{align}
q^{2}>\max\left\lbrace\frac{-C(\rho^{\ast},\theta^{\ast})-B(\rho^{\ast},\theta^{\ast})}{A(\rho^{\ast},\theta^{\ast})g^{2}(\rho^{\ast},\theta^{\ast})},~\frac{2}{g^{2}(\rho^{\ast},\theta^{\ast})}\right\rbrace\label{eq:qvalues}
\end{align}
and using that $A<0$, it follows that $Aa_{1}^{2}+B<-C$, and so
\begin{align*}
\mathcal{P}(a_{1})=a_{1}^{2}\left(Aa_{1}^{2}+B\right)+C<-C.
\end{align*}
Hence, \eqref{eq:discriminant} is negative and according with Lemma \ref{fourthroots}, there are two complex roots of \eqref{eq:dquartic}. Now, let $v\in\mathbb{R}$, take $q$ satisfying \eqref{eq:qvalues} and set $U^{\ast}:=(\rho^{\ast},v,\theta^{\ast},q)\in\mathcal{O}\subset\mathbb{R}^{4}$. Then, by the previous argument, the characteristic polynomial \eqref{eq:cp1d} has complex roots at $U^{\ast}$. This proves the result.
\end{proof}
\section{Three dimensional system}
Set $d=3$ and observe that $N=8$. In this case, $A^{0}(U)$ is a diagonal matrix given as
\begin{align}
A^{0}(U)=\left(\begin{array}{cccc}
	1&&&\\
	&\rho\mathbb{I}_{3}&&\\
	&&\rho e_{\theta}&\\
	&&&\tau\mathbb{I}_{3}
\end{array}\right),\label{eq:A03d}
\end{align}
where $\mathbb{I}_{3}$ denotes the identity matrix of order $3\times 3$ and all the empty spaces refer to zero block matrices of the appropriate sizes. We have that
\small
\begin{align}
\label{eq:Axi}
A_{\lambda,\nu}(\xi;U)=\left(\begin{array}{cccccccc}
	\xi\cdot v&\xi_{1}\rho&\xi_{2}\rho&\xi_{3}\rho&0&0&0&0\\
	\xi_{1}p_{\rho}&\rho\xi\cdot v&0&0&\xi_{1}p_{\theta}&0&0&0\\
	\xi_{2}p_{\rho}&0&\rho\xi\cdot v&0&\xi_{2}p_{\theta}&0&0&0\\
	\xi_{3}p_{\rho}&0&0&\rho\xi\cdot v&\xi_{3}p_{\theta}&0&0&0\\
	0&\xi_{1}\theta p_{\theta}&\xi_{2}\theta p_{\theta}&\xi_{3}\theta p_{\theta}&\rho e_{\theta}\xi\cdot v&\xi_{1}&\xi_{2}&\xi_{3}\\
	0&&&&\xi_{1}\kappa&\tau\xi\cdot v&0&0\\
	0&&\tau\mathcal{Q}_{\lambda,\nu}(\xi;q)&&\xi_{2}\kappa&0&\tau\xi\cdot v&0\\
	0&&&&\xi_{3}\kappa&0&0&\tau\xi\cdot v\\
\end{array}\right),
\end{align}
\normalsize
where, for each $\xi\in\mathbb{S}^{2}$, the sub-block matrix $\mathcal{Q}_{\lambda,\nu}(q;\xi)$ is of order $3\times 3$ and given as
\small
\begin{equation}
\mathcal{Q}_{\lambda,\nu}(q;\xi)=\left(\begin{array}{ccc}
	\gamma\xi_{1}q_{1}+\lambda^{-}\left(\xi_{2}q_{2}+\xi_{3}q_{3}\right)&\lambda^{+}\xi_{1}q_{2}+\nu\xi_{2}q_{1}&\lambda^{+}\xi_{1}q_{3}+\nu\xi_{3}q_{1}\\
	\lambda^{+}\xi_{2}q_{1}+\nu\xi_{1}q_{2}&\gamma\xi_{2}q_{2}+\lambda^{-}\left(\xi_{1}q_{1}+\xi_{3}q_{3}\right)&\lambda^{+}\xi_{2}q_{3}+\nu\xi_{3}q_{2}\\
	\lambda^{+}\xi_{3}q_{1}+\nu\xi_{1}q_{3}&\lambda^{+}\xi_{3}q_{2}+\nu\xi_{2}q_{3}&\gamma\xi_{3}q_{3}+\lambda^{-}\left(\xi_{1}q_{1}+\xi_{2}q_{2}\right)
\end{array}\right),
\label{eq:23}
\end{equation}
\normalsize
where $\gamma:=\lambda+\nu$, $\lambda^{+}:=\frac{\lambda}{2}+\frac{1}{2}$ and $\lambda^{-}:=\frac{\lambda}{2}-\frac{1}{2}$, and
\begin{align}
Q(U)=\left(0,0,0,0,0,q_{1},q_{2},q_{3}\right)^{\top}.\label{eq:QU}
\end{align}
By using formula \eqref{eq:detblock}, we can compute the characteristic polynomial for (\ref{eq:21}), yielding
\begin{equation}
\det\left(A_{\lambda\nu}(\xi;U)-\eta A^{0}(\xi;U)\right)=\rho^{3}\tau^{2}(\xi\cdot v-\eta)^{4}P_{\lambda,\nu}(\xi,U;\eta),
\label{eq:59}
\end{equation}
where 
\begin{equation}
P_{\lambda,\nu}(\xi,U;\eta)=\rho e_{\theta}\tau(\xi\cdot v-\eta)^{4}-\left(\tau\rho e_{\theta}p_{\rho}+\frac{\tau\theta p_{\theta}^{2}}{\rho}+\kappa\right)(\xi\cdot v-\eta)^{2}+\frac{p_{\theta}}{\rho}h_{\lambda,\nu}(\xi;q)(\xi\cdot v-\eta)+\kappa p_{\rho}
\label{eq:510}
\end{equation}
and $h_{\lambda,\nu}(\xi;q):=\left(Q_{\lambda,\nu}(\xi;q)\xi\right)\cdot\xi$.
\begin{lema}
\label{Qnontrivial}
Fix $\lambda,\nu\in\mathbb{R}$. The mapping, $\mathbb{S}^{2}\times\mathbb{R}^{3}\ni(\xi;q)\mapsto\mathcal{Q}_{\lambda,\nu}(\xi;q)\in\mathbb{M}_{3\times3}$, is non-trivial.
\end{lema}
\begin{proof}
Assume otherwise, that is, for all $q\in\mathbb{R}^{3}$ and all $\xi\in\mathbb{S}^{2}$, $\mathcal{Q}_{\lambda,\nu}(\xi;q)=0$. Take $\xi=(1,0,0)$ and $q=(q_{1},0,0)$ with $q_{1}\neq 0$. Then, by \eqref{eq:23}, $\gamma q_{1}=0$ and $\lambda^{-}q_{1}=0$. This means that $\lambda=1$ and $\nu=-1$. On the other hand, take $\xi=(1,0,0)$ and $q=(0,q_{2},0)$ with $q_{2}\neq 0$. Then, \eqref{eq:23} implies that $\nu q_{2}=0$, a contradiction.
\end{proof}
\begin{lema}
\label{hnull}
Fix $\lambda,\nu\in\mathbb{R}$. The mapping, $\mathbb{S}^{2}\times\mathbb{R}^{3}\ni(\xi,q)\mapsto h_{\lambda,\nu}(\xi;q)\in\mathbb{R}$, is null if and only if $\lambda+\nu=0$.
\end{lema}
\begin{proof}
First observe that for any $q\in\mathbb{R}^{3}$ and $\xi\in\mathbb{S}^{2}$ it holds that 
\begin{align*}
Q_{\lambda,\nu}(\xi;q)\xi\cdot \xi=\gamma\left(\xi_{1}^{3}q_{1}+\xi_{2}^{3}q_{2}+\xi_{3}^{3}q_{3}+\xi_{1}^{2}\xi_{2}q_{2}+\xi_{1}^{2}\xi_{3}q_{3}+\xi_{1}\xi_{2}^{2}q_{1}+\xi_{1}\xi_{3}^{2}q_{1}+\xi_{3}\xi_{2}^{2}q_{3}+\xi_{2}\xi_{3}^{2}q_{2}\right),
\end{align*}
which yields the formula
\begin{align}
\label{eq:hnull}
h_{\lambda,\nu}(q;\xi)=\gamma|\xi|^{2}\xi\cdot q\quad\forall\quad q\in\mathbb{R}^{3},\quad\xi\in\mathbb{S}^{2}.
\end{align}
Then, if $\gamma=\lambda+\nu=0$, \eqref{eq:hnull} implies that $h_{\lambda,\nu}(q;\xi)=0$ for any $q\in\mathbb{R}^{3}$ and $\xi\in\mathbb{S}^{2}$. On the other hand, assume that $h_{\lambda,\nu}(\xi;q)=0$ for all $\xi\in\mathbb{S}^{2}$ and $q\in\mathbb{R}^{3}$. If so, it must be true for $q_{0}\in\mathbb{R}^{3}$ and $\xi_{0}\in\mathbb{S}^{2}$ such that $\xi_{0}\cdot q_{0}\neq 0$. Then, by \eqref{eq:hnull}, we have that $\gamma|\xi_{0}|^{2}\xi_{0}\cdot q_{0}=0$, which implies that $\gamma=\lambda+\nu=0$.
\end{proof}
\begin{theo}
Set $d=3$ and consider the $(\lambda,\nu)$-quasilinear system defined by \eqref{eq:A03d}, \eqref{eq:Axi}, \eqref{eq:23} and \eqref{eq:QU}. If there is $(\rho^{\ast},\theta^{\ast})\in\mathcal{D}$ such that $\gamma^{\ast}:=\lambda(\rho^{\ast},\theta^{\ast})+\nu(\rho^{\ast},\theta^{\ast})\neq 0$ then, the $(\lambda,\nu)$-quasilinear system is not hyperbolic. 
\end{theo}
\begin{proof}
For any $q\in\mathbb{R}^{3}\setminus\{0\}$ set $\xi_{q}=\frac{q}{|q|}\in\mathbb{S}^{2}$. Then, by \eqref{eq:hnull} and Lemma \ref{hnull}, $h_{\lambda,\nu}(\xi_{q};q)=\gamma|q|\neq 0$. Take $v\in\mathbb{R}^{3}$ and define the state $U^{\ast}:=(\rho^{\ast},v,\theta^{\ast},q)$. Set $p_{\rho}^{\ast}=p_{\rho}(\rho^{\ast},\theta^{\ast})$, $p_{\theta}^{\ast}=p_{\theta}(\rho^{\ast},\theta^{\ast})$, $e_{\theta}^{\ast}=e_{\theta}(\rho^{\ast},\theta^{\ast})$ and $\kappa^{\ast}=\kappa(\rho^{\ast},\theta^{\ast})$. Consider $P_{\lambda,\nu}(\xi_{q},U^{\ast};\eta)=0$, which according with \eqref{eq:510} is equivalent to 
\begin{align}
\rho^{\ast} e_{\theta}^{\ast}\tau(\xi_{q}\cdot v-\eta)^{4}-\left(\tau\rho^{\ast} e_{\theta}^{\ast}p_{\rho}^{\ast}+\frac{\tau\theta^{\ast} (p_{\theta}^{\ast})^{2}}{\rho^{\ast}}+\kappa^{\ast}\right)(\xi_{q}\cdot v-\eta)^{2}+\frac{p_{\theta}^{\ast}}{\rho^{\ast}}\gamma^{\ast}|q|(\xi_{q}\cdot v-\eta)+\kappa^{\ast} p_{\rho}^{\ast}=0.
\label{eq:P3d}
\end{align}
Notice that, \eqref{eq:P3d} is the same as the polynomial given in \eqref{eq:cp1d} but with $|q|$ and $v\cdot\xi_{q}$ replacing $q$ and $v$, respectively. Thus, by \eqref{eq:qvalues}, with $|q|^{2}$ instead of $q^{2}$, it holds that \eqref{eq:P3d} has complex roots. Therefore, the $(\lambda,\nu)$-quasilinear system is not hyperbolic if $\lambda+\nu\neq 0$ on $\mathcal{D}$.
\end{proof}
\section{The Cattaneo-Christov-Jordan system}
Consider the coupling between \eqref{eq:1}-\eqref{eq:3} and the material formulation of the Cattaneo heat flux model, namely,
\begin{equation}
\tau\left(q_{t}+u\cdot\nabla q\right)+q=-\kappa\nabla\theta,
\label{eq:4}
\end{equation} 
This system can be written in the quasilinear form \eqref{eq:21} where for each $U\in\mathcal{O}$, $A^{0}(U)$ and $Q(U)$ are given by \eqref{eq:A03d} and \eqref{eq:QU} respectively, however, instead of $A_{\lambda,\nu}(\xi;U)$ we have the following symbol
\small
\begin{align}
\label{eq:AxiCCJ}
A(\xi;U)=\left(\begin{array}{cccccccc}
	\xi\cdot v&\xi_{1}\rho&\xi_{2}\rho&\xi_{3}\rho&0&0&0&0\\
	\xi_{1}p_{\rho}&\rho\xi\cdot v&0&0&\xi_{1}p_{\theta}&0&0&0\\
	\xi_{2}p_{\rho}&0&\rho\xi\cdot v&0&\xi_{2}p_{\theta}&0&0&0\\
	\xi_{3}p_{\rho}&0&0&\rho\xi\cdot v&\xi_{3}p_{\theta}&0&0&0\\
	0&\xi_{1}\theta p_{\theta}&\xi_{2}\theta p_{\theta}&\xi_{3}\theta p_{\theta}&\rho e_{\theta}\xi\cdot v&\xi_{1}&\xi_{2}&\xi_{3}\\
	0&&&&\xi_{1}\kappa&\tau\xi\cdot v&0&0\\
	0&&\mathbb{O}_{3\times3}&&\xi_{2}\kappa&0&\tau\xi\cdot v&0\\
	0&&&&\xi_{3}\kappa&0&0&\tau\xi\cdot v\\
\end{array}\right),
\end{align}
\normalsize
where, $\mathbb{O}_{3\times 3}$ denotes the zero matrix of order $3\times 3$. This system was first proposed by Christov and Jordan \cite{chjo}. They showed that \eqref{eq:4} is Galilean invariant. We refer to the quasilinear system defined by \eqref{eq:A03d}, \eqref{eq:QU} and \eqref{eq:AxiCCJ} as the Cattaneo-Christov-Jordan system. Morro points out that \eqref{eq:4} is not objective and therefore for no value of $\lambda$ and $\nu$ this equation can be deduced from \eqref{eq:objder}. This coincides with Lemma \ref{Qnontrivial}. Nonetheless, formally, we can understand the Cattaneo-Christov-Jordan system as a quasilinear system with $\mathbb{O}_{3\times 3}$ instead of $\mathcal{Q}_{\lambda,\nu}(\xi;q)$. The characteristic polynomial of the Cattaneo-Christov-Jordan system has the form 
\begin{equation}
\det\left(A_{\lambda\nu}(\xi;U)-\eta A^{0}(\xi;U)\right)=\rho^{3}\tau^{2}(\xi\cdot v-\eta)^{4}P_{0}(\xi,U;\eta),
\label{eq:CCJpolynomial}
\end{equation}
where 
\begin{equation}
P_{0}(\xi,U;\eta)=\rho e_{\theta}\tau(\xi\cdot v-\eta)^{4}-\left(\tau\rho e_{\theta}p_{\rho}+\frac{\tau\theta p_{\theta}^{2}}{\rho}+\kappa\right)(\xi\cdot v-\eta)^{2}+\kappa p_{\rho}.
\label{eq:realroots}
\end{equation}
The characteristic roots of the three dimensional Cattaneo-Christov-Jordan system are real and given by 
\begin{align}
\eta_{0}(\xi;U)=\xi\cdot v,&\nonumber\\
\eta_{1}(\xi;U)=\xi\cdot v+\frac{1}{\sqrt{2}}\sqrt{r(\rho,\theta)+m(\rho,\theta)},\quad \eta_{2}(\xi;U)&=\xi\cdot v+\frac{1}{\sqrt{2}}\sqrt{r(\rho,\theta)-m(\rho,\theta)},\label{eq:cs0}\\
\eta_{3}(\xi;U)=\xi\cdot v-\frac{1}{\sqrt{2}}\sqrt{r(\rho,\theta)+m(\rho,\theta)},\quad\eta_{4}(\xi;U)&=\xi\cdot v-\frac{1}{\sqrt{2}}\sqrt{r(\rho,\theta)-m(\rho,\theta)},\nonumber
\end{align}
where, for each $\rho,\theta>0$ we have set
\begin{align*}
r(\rho,\theta):=\left(p_{\rho}+\frac{\theta p_{\theta}^{2}}{\rho^{2}e_{\theta}}+\frac{\kappa}{\rho e_{\theta}\tau}\right)\quad\mbox{and}\quad m(\rho,\theta):=\sqrt{\left(p_{\rho}+\frac{\theta p_{\theta}^{2}}{\rho^{2}e_{\theta}}+\frac{\kappa}{\rho e_{\theta}\tau}\right)^{2}-\frac{4p_{\rho}\kappa}{\rho e_{\theta}\tau}}.
\end{align*}
Observe that, by \eqref{eq:CCJpolynomial}, $\eta_{0}(\xi;U)$ is a root of algebraic multiplicity four, and it is easy to see that $\eta_{3}<\eta_{4}<\eta_{2}<\eta_{1}$ (see \cite[Sections 3 and 4.1]{felipe1}). Moreover, this system is Friedrichs symmetrizable \cite[Section 4]{felipe1}, which in turn implies that it is hyperbolic and thus, its Cauchy problem is locally well-posed in $L^{2}$ (see, \cite{katosym}, \cite{otto} and \cite{lax}, for example).
\begin{theo}
\label{ccjroots}
Consider a $(\lambda,\nu)$-quasilinear system for which $\lambda+\nu=0$ on $\mathcal{D}$. Then, the characteristic speeds of such system are real and coincide with the characteristic speeds of the Cattaneo-Christov-Jordan system given in \eqref{eq:cs0}.
\end{theo}
\begin{proof}
Since $\lambda+\nu=0$, Lemma \ref{hnull} assures that $h_{\lambda,\nu}(\xi;q)=0$ for all $\xi\in\mathbb{S}^{2}$ and $q\in\mathbb{R}^{3}$. Therefore, 
\begin{align*}
P_{\lambda,\nu}(\xi,U;\eta)=P_{0}(\xi,U;\eta)\quad\mbox{for all}\quad\xi\in\mathbb{S}^{2},~U\in\mathbb{R}^{8},~\eta\in\mathbb{R},
\end{align*}
and the result follows by comparing \eqref{eq:59} and \eqref{eq:CCJpolynomial}.
\end{proof}
\begin{rem}
As it was pointed out in \cite{felipe1}, a quasilinear system of the form \eqref{eq:21}, with real characteristic speeds, is not necessarily hyperbolic. Such is the case of the Cattaneo-Christov system, since in this case, $(\lambda,\nu)=(-1,1)$ and so $\lambda+\nu=0$ \textup{(}see \cite[Theorem 3.5]{felipe1}\textup{)}.
\end{rem}
\begin{theo}
Let assumptions \textup{(\textbf{C1})-(\textbf{C3})} be satisfied. Consider a $(\lambda,\nu)$-quasilinear system \eqref{eq:21} with the property that $\lambda+\nu=0$ on $\mathcal{D}$. This system is hyperbolic if and only if $(\lambda,\nu)=(1,-1)$.
\end{theo}
\begin{proof}
Assume that $\lambda+\nu=0$ on $\mathcal{D}$. According with the definition \ref{hyperbolic}, we have to show that the only values of $(\lambda,\nu)$ for which the eigenvalue problem \eqref{eq:hypdet} has a complete set of eigenvectors are $(\lambda,\nu)=(-1,1)$. Since $\lambda+\nu=0$, Theorem \ref{ccjroots} assures that the eigenvalues of the $(\lambda,\nu)$-quasilinear system are given by \eqref{eq:cs0}. Let $\{V_{j}\}_{j=1}^{4}$ be the eigenvectors associated with the eigenvalues $\{\eta_{j}\}_{j=1}^{4}$, respectively. By \textbf{(C1)}-\textbf{(C2)}, these eigenvalues are different and so, the set of eigenvectors $\{V_{j}\}_{j=1}^{4}$ is linearly independent. Since $\eta_{0}(\xi;U)$ is different from the other roots, the $(\lambda,\nu)$-quasilinear system will be hyperbolic if and only if, for every $\xi\in\mathbb{S}^{2}$ and $U\in\mathcal{O}$, the geometric multiplicity of $\eta_{0}(\xi;U)$ equals four. Let $\eta=\eta_{0}$ and $A(\xi; U)=A_{\lambda,\nu}(\xi;U)$ in \eqref{eq:hypdet} and set $V=(V_{1},V_{2},..,V_{8})^{\top}$ to obtain the algebraic system of equations, 
\begin{align}
\xi\cdot V^{\prime}&=0,\label{eq:hyp1}\\
p_{\rho}V_{1}+p_{\theta}V_{5}&=0,\label{eq:hyp2}\\
\xi\cdot V^{\prime\prime}&=0,\label{eq:hyp3}\\
\tau\mathcal{Q}_{\lambda,\nu}(q;\xi)V^{\prime}+\kappa V_{5}\xi&=0,\label{eq:hyp4}
\end{align}
where $V^{\prime}=(V_{2},V_{3},V_{4})^{\top}$ and $V^{\prime\prime}=(V_{6},V_{7},V_{8})^{\top}$. From \eqref{eq:hyp3}, it follows the existence of exactly two linearly independent solutions, say $(V_{6}^{1},V_{7}^{1},V_{8}^{1})^{\top}$ and $(V_{6}^{2},V_{7}^{2},V_{7}^{2})^{\top}$. Then, we can take the vectors
\begin{align}
(0,0,0,0,0,V_{6}^{i},V_{7}^{i},V_{8}^{i})^{\top},\quad\mbox{for}~i=1,2,\label{eq:hypvectors}
\end{align}
as two linearly independent solutions of \eqref{eq:hyp1}-\eqref{eq:hyp4}. Therefore, in order to comply with the hyperbolicity, equations \eqref{eq:hyp1}, \eqref{eq:hyp2} and \eqref{eq:hyp4}, must have two more linearly independent solutions for each $\xi\in\mathbb{S}^{2}$ and $U=(\rho,v,\theta,q)\in\mathcal{O}$. Observe that, for each $\xi\in\mathbb{S}^{2}$, \eqref{eq:hyp1} implies that $V^{\prime}\in\{\xi\}^{\perp}$, a two dimensional space. Hence, if $\{V_{\xi},W_{\xi}\}$ is a basis of $\{\xi\}^{\perp}$, $V^{\prime}$ is of the form
\begin{align}
\label{eq:hyp5}
V^{\prime}=\alpha_{1}V_{\xi}+\alpha_{2}W_{\xi},\quad\mbox{for some}~\alpha_{1},\alpha_{2}\in\mathbb{R}.
\end{align}
If we multiply \eqref{eq:hyp4} by $\xi$ and use \eqref{eq:hyp5} we obtain that
\begin{align}
\label{eq:hyp6}
V_{5}=-\frac{\tau}{\kappa}\left\lbrace\alpha_{1}\left(\mathcal{Q}_{\lambda,\nu}(q;\xi)V_{\xi}\right)\cdot\xi+\alpha_{2}\left(\mathcal{Q}_{\lambda,\nu}(q;\xi)W_{\xi}\right)\cdot\xi\right\rbrace\quad\mbox{for each}~\xi\in\mathbb{S}^{2}.
\end{align}
Let $\xi=(\xi_{1},\xi_{2},\xi_{3})\in\mathbb{S}^{2}$ and assume that $\xi_{1}\neq 0$. Then, the vectors
\begin{align*}
V_{\xi}=\left(-\xi_{2}, \xi_{1}, 0\right)^{\top}\quad\mbox{and}\quad W_{\xi}=\left(-\xi_{3}, 0, \xi_{1}\right)^{\top}
\end{align*}
are linearly independent and form a basis of $\{\xi\}^{\perp}$. Hence, by \eqref{eq:hyp6} 
\begin{align}
\label{eq:hyp7}
V_{5}=-\frac{\tau}{\kappa}\lambda^{+}\left\lbrace\alpha_{1}\left(\xi_{1}q_{2}-\xi_{2}q_{1}\right)+\alpha_{2}\left(\xi_{1}q_{3}-\xi_{3}q_{1}\right)\right\rbrace\quad\mbox{for any}~q=(q_{1},q_{2},q_{3})\in\mathbb{R}^{3}.
\end{align}
By using \eqref{eq:hyp5} and \eqref{eq:hyp7} into \eqref{eq:hyp4} we obtain the equations
\begin{align}
\alpha_{1}\left[-\lambda^{-}(\xi_{2}^{2}q_{2}+\xi_{2}\xi_{3}q_{3})+(\nu+\lambda^{+})\xi_{1}\xi_{2}q_{1}\right]&+\alpha_{2}\left[-\lambda^{-}(\xi_{2}\xi_{3}q_{2}+\xi_{3}^{2}q_{3})+(\nu+\lambda^{+})\xi_{1}\xi_{3}q_{1}\right]=0,\label{eq:hyp8}\\
\alpha_{1}\left[-(\nu+\lambda^{+})\xi_{1}\xi_{2}q_{2}\right.&+\left.\lambda^{-}(\xi_{1}^{2}q_{1}+\xi_{1}\xi_{3}q_{3})\right]=0,\label{eq:hyp9}\\
\alpha_{2}\left[-(\nu+\lambda^{+})\xi_{1}\xi_{3}q_{3}\right.&+\left.\lambda^{-}(\xi_{1}^{2}q_{1}+\xi_{1}\xi_{2}q_{2})\right]=0,\label{eq:hyp10}
\end{align}
valid for any $\xi\in\mathbb{S}^{2}$ with $\xi_{1}\neq 0$. Take $\overline{\xi}=(\overline{\xi}_{1},\overline{\xi}_{2},\overline{\xi}_{3})\in\mathbb{S}^{2}$ such that $\overline{\xi}_{i}\neq 0$ for all $i=1,2,3$ and $\overline{q}\in\mathbb{R}^{3}$ with the property that $\overline{\xi}\cdot\overline{q}\neq 0$. Assume the existence of $(\rho^{\prime},\theta^{\prime})\in\mathcal{D}$ such that $(\lambda,\nu)(\rho^{\prime},\theta^{\prime})\neq(1,-1)$. Then, $\left(\nu+\lambda^{+}\right)(\rho^{\prime},\theta^{\prime})=-\lambda^{-}(\rho^{\prime},\theta^{\prime})\neq 0$ and \eqref{eq:hyp8}-\eqref{eq:hyp10} can be rewritten as 
\begin{align*}
\alpha_{1}\overline{\xi}_{2}\left(\overline{\xi}\cdot\overline{q}\right)=-\alpha_{2}\overline{\xi}_{3}\left(\overline{\xi}\cdot\overline{q}\right),\quad
\alpha_{1}\overline{\xi}_{1}\left(\overline{\xi}\cdot\overline{q}\right)=0,\quad
\alpha_{2}\overline{\xi}_{1}\left(\overline{\xi}\cdot\overline{q}\right)=0.
\end{align*}
By the particular choices of $\overline{\xi}$ and $\overline{q}$, it follows that $\alpha_{1}=\alpha_{2}=0$ and thus, $V^{\prime}=0$. This implies that $V_{1}=V_{5}=0$. Therefore, if $(\lambda,\nu)\neq(1,-1)$ we can find wavenumbers ($\overline{\xi}$) and states ($\overline{U}=(\rho^{\prime},v,\theta^{\prime},\overline{q})^{\top}\in\mathcal{O}$) such that the geometric multiplicity of the eigenvalue $\eta_{0}(\overline{\xi};\overline{U})$ equals two. This means that the $(\lambda,\nu)$-quasilinear system is not hyperbolic if $(\lambda,\nu)\neq(1,-1)$.\\
Now assume that $(\lambda,\nu)=(1,-1)$ on $\mathcal{D}$. If $\xi=(\xi_{1},\xi_{2},\xi_{3})\in\mathbb{S}^{2}$ is such that $\xi_{1}\neq 0$, $V_{5}$ is given by \eqref{eq:hyp7} and \eqref{eq:hyp4} is equivalent to equations \eqref{eq:hyp8}-\eqref{eq:hyp10}. Moreover, since $\lambda^{+}=1$ and $\lambda^{-}=0$, equations \eqref{eq:hyp8}-\eqref{eq:hyp10} are trivially satisfied for any choice of $\alpha_{1},\alpha_{2}\in\mathbb{R}$. In particular, this means that $V_{\xi}$ and $W_{\xi}$ are linearly independent solutions of equations \eqref{eq:hyp1} and \eqref{eq:hyp4}, where $V_{5}$ is given by \eqref{eq:hyp7}. Consequently, the vectors 
\begin{align}
\mathcal{V}(\xi)=\left(-\frac{p_{\theta}}{p_{\rho}}V_{5}, V_{\xi}, V_{5},0,0,0\right)^{\top},\quad \mathcal{W}(\xi)=\left(-\frac{p_{\theta}}{p_{\rho}}V_{5}, W_{\xi}, V_{5},0,0,0\right)^{\top},\label{eq:hypvectors2}
\end{align}
are solutions of equations \eqref{eq:hyp1}-\eqref{eq:hyp4} and together with \eqref{eq:hypvectors} form a linearly independent set. Therefore, the geometric multiplicity of $\eta_{0}(\xi;U)$ equals four, for any $\xi\in\mathbb{S}^{2}$ with $\xi_{1}\neq 0$ and $U\in\mathcal{O}$. If, on the other hand, $\xi\in\mathbb{S}^{2}$ is such that $\xi_{2}\neq 0$, we take the linearly independent vectors,
\begin{align*}
V_{\xi}=\left(\xi_{2}, -\xi_{1}, 0\right)^{\top}\quad\mbox{and}\quad W_{\xi}=\left(0, -\xi_{3}, \xi_{2}\right)^{\top},
\end{align*}
as solutions of \eqref{eq:hyp1}. In this case, 
\begin{align*}
V_{5}=-\frac{\tau}{\kappa}\left\lbrace\alpha_{1}\left(\xi_{2}q_{1}-\xi_{1}q_{2}\right)+\alpha_{2}\left(\xi_{2}q_{3}-\xi_{3}q_{2}\right)\right\rbrace\quad\mbox{for any}~q=(q_{1},q_{2},q_{3})\in\mathbb{R}^{3}.
\end{align*}
By setting $\mathcal{V}(\xi)$ and $\mathcal{W}(\xi)$ as in \eqref{eq:hypvectors2}, the conclusion follows. Finally, if $\xi=(0,0,1)$ take the canonical vectors $\hat{e}_{1}$ and $\hat{e}_{2}$ as solutions of \eqref{eq:hyp4}, that is, $V_{\xi}=\hat{e}_{1}$ and $W_{\xi}=\hat{e}_{2}$. Then, $V_{5}=-\frac{\tau}{\kappa}\left(\alpha_{1}q_{1}+\alpha_{2}q_{2}\right)$ and we proceed as before. Therefore, if $(\lambda,\nu)=(1,-1)$, the geometric multiplicity of the eigenvalue $\eta_{0}(\xi,U)$ equals four for any choice of $\xi\in\mathbb{S}^{2}$ and $U\in\mathcal{O}$ and thus, the $(1,-1)$-quasilinear system is hyperbolic.
\end{proof}
\section{Final comments and conclusions}
In this work we have shown that the only case in which the coupling between \eqref{eq:1}-\eqref{eq:3} and the frame indifferent formulation for the heat flux \eqref{eq:objder} yields a hyperbolic system of equations is when we take $(\lambda,\nu)$ as the constant value functions $(1,-1)$ on $\mathcal{D}$. In particular, this implies that for a constant state $V_{c}\in\mathcal{O}$, the Cauchy problem for the linear equation
\begin{equation*}
A^{0}(V_{c})U_{t}+A^{i}_{1,-1}(V_{c})\partial_{i}U+Q(V_{c})=0,
\end{equation*}
is well posed in $L^{2}$, as consequence of the existence of energy estimates \cite[Theorem 3.1.2]{serre}. Moreover, if $V_{c}\in\mathcal{O}$ is taken as constant or constant outside a compact set and since the characteristic speeds of this model (see Theorem \ref{ccjroots}) have constant algebraic multiplicity with respect to $\xi\in\mathbb{S}^{2}$, the finite speed of propagation holds true \cite[Theorem 2.11]{benzoni}. Now, for the general $(1,-1)$-quasilinear system \eqref{eq:21}, although its hyperbolicity is necessary for the existence of $L^{2}$ energy estimates, it is not sufficient \cite{gstran}. Typically, when the quasilinear system \eqref{eq:21} is derived from a system of conservation laws, the existence of a convex entropy is equivalent to the existence of a symmetrizer (as the Hessian of the entropy) (cf. \cite{daf}, \cite{serre}). Given the non-conservative structure of the $(1,-1)$-heat flux model, it is unknown for the author if a Friedrichs symmetrizer exists for this quasilinear system of equations in more than one space dimensions. Due to the presence of the skew matrix $\mathcal{Q}_{1,-1}(\xi;q)$, a diagonal symmetrizer as in the Cattaneo-Christov-Jordan system, doesn't work for this case when $d=3$.\\
By proceeding as in \cite[Section 3.1]{christov}, it is easy to show that the $(1,-1)$-heat flux model is \emph{irreducible} in several space dimensions, in the sense that, together with the material invariant form of the internal energy equation, it is impossible to derive a single equation for the temperature field $\theta$. Observe that the one dimensional version of this model coincides with the one dimensional Cattaneo-Christov-Jordan system. Therefore, it is locally well-posed in $L^{2}$ and satisfies the finite speed of propagation property (see \cite{katosym} and \cite[Theorem 2.11]{benzoni}.
\section*{Acknowledgments}
Thanks to Angelo Morro for clarifying some aspects of objective rates. This work was supported by CONACyT (Mexico) through a postdoctoral fellowship under grant A1-S-10457.

\bibliographystyle{plain} 
\bibliography{feVSkawabiblio}

\end{document}